\documentclass{base_class}

\title[]{On Hodge-Riemann Relations \\for Translation-Invariant Valuations} 

\author{Jan Kotrbat{\'y}}\thanks{The author was supported by DFG grant WA3510/1-1 and by FWF Grant P31448-N35.}
 
\email{kotrbaty@math.uni-frankfurt.de}
\address{Institut für Mathematik, Goethe-Universit\"at Frankfurt, Robert-Mayer-Str. 10, 60629 Frankfurt, Germany}
\date{\today}


\begin{document}

\maketitle

\begin{abstract}
The Alesker product turns the space of smooth translation-invariant valuations on convex bodies into a commutative associative unital algebra, satisfying Poincar\'e duality and the hard Lefschetz theorem. In this article, a version of the Hodge-Riemann relations for the Alesker algebra is conjectured, and the conjecture is proved in two particular situations: for even valuations, and for 1-homogeneous valuations. The latter result is then used to deduce a special case of the Aleksandrov-Fenchel inequality. Finally, mixed versions of the hard Lefschetz theorem and of the Hodge-Riemann relations are conjectured, and it is shown that the Aleksandrov-Fenchel inequality follows from the latter in its full generality.

\end{abstract}

\section{Introduction}

A {\it valuation} is---at least for the purpose of the text that follows---a finitely additive measure on convex bodies in the Euclidean space. Of particular interest is the space $\Val$ of valuations that are continuous and translation invariant; it was conjectured by Peter McMullen \cite{McMullen80} that $\Val$ is {\it just the weak completion of the linear space of mixed volumes}. McMullen's conjecture was proved by Semyon Alesker \cite{Alesker01}, in fact, in much greater generality of what has become known as the {\it irreducibility theorem}. The salient feature of Alesker's result is that it acted as a catalyst for a variety of further developments, among the most important of which is certainly the discovery of a natural product of valuations.

The {\it Alesker product} \cite{Alesker04} is defined on the dense subspace $\Val^\infty\subset\Val$ of smooth valuations (we refer to \S\ref{ss:smooth} for details), turning it into a commutative associative unital algebra that is graded by the degree of homogeneity and by the parity of a valuation. In this connection, let us recall an important result of P. McMullen according to which the degree of homogeneity is a non-negative integer less than or equal the dimension $n$ of the underlying Euclidean space. Remarkably, the Alesker algebra moreover satisfies a {\it Poincar\'e-type duality}, and even a version of the {\it hard Lefschetz theorem} for multiplication by the first intrinsic volume. Dual to the product is another natural multiplicative structure on $\Val^\infty$, the {\it Bernig-Fu convolution} introduced in \cite{BernigFu06}. The product and the convolution are further intertwined by a {\it Fourier-type transform}, which was constructed by S. Alesker in \cite{Alesker03hard,Alesker11}. It is apropos to point out that all these structures behave naturally with respect to the `standard' examples of valuations, such as to the Euler characteristic and the Lebesque measure, or, more generally, to the intrinsic and mixed volumes.

A fundamental connection between the product of valuations and integral geometry was found by Bernig and Fu \cite{Fu06,BernigFu06}: It turns out that the array of classical kinematic formulas for the intrinsic volumes---to which great attention was paid throughout the twentieth century, in particular, in works of Blaschke, Chern, Hadwiger, Federer, or Santal\'o---precisely corresponds to the structure of the (finite-dimensional) subalgebra $\Val^{\SO(n)}\subset\Val^\infty$ of $\SO(n)$-invariant valuations. Furthermore, the same correspondence applies to any closed subgroup $G\subset\SO(n)$ acting transitively on spheres, and can even be extended to a general isotropic space \cite{AleskerBernig12}. Application of this principle resulted in a comprehensive understanding of Hermitian integral geometry \cite{Bernig09,BernigFu11,BernigSolanesFu14,Wannerer14module,Wannerer14}, and led to substantial progress in the other---quaternionic and octonionic---spaces \cite{Bernig11,BernigSolanes17,SolanesWannerer21,KotrbatyPhD}.

\subsection{Main results}

The aim of this article is to explore a new, yet very natural feature of the algebra of smooth valuations. Our results indicate that Alesker's algebraic theory of valuations is even more powerful and encompassing than so far believed. Specifically, we show that not only it explains and extends integral geometry of intrinsic volumes, but it also appears capable of subsuming another main pillar of classical convex geometry---the Aleksandrov inequalities---and thus merging these two seemingly separated topics into a unified concept.

To motivate our work, let us recall a fundamental result from the cohomology theory of compact Kähler manifolds that is closely related to the Poincar\'e duality (PD) and the hard Lefschetz theorem (HL): A legitimate question to ask is what is the signature of the Hermitian pairing obtained by composing the Poincar\'e pairing with the Lefschetz map. Since the latter is self-adjoint with respect to the former, it is enough to restrict the problem to primitive elements. According to the {\it Hodge-Riemann relations} (HR), the pairing is then either positive or negative definite, depending on the degree of the cohomology class (see, e.g., \S3 of \cite{Huybrechts2005}). Remarkably, some version of the so-called {\it Kähler package} PD--HL--HR has appeared in numerous different situations outside of Kähler geometry \cite{AdiprasitoHuhKatz18,Karu04,McMullen93,EliasWilliamson14,HuhWang17,BraendenHuh20}, and very often proved to be the key to solving a long-standing open problem (see also Huh's lucid account \cite{Huh18}).

In this article, a step is taken towards completing the Kähler package for the algebra $\Val^\infty$; namely, a version of the Hodge-Riemann relations is proved in two particular situations. 

In order to recast the above problem in the language of the Alesker product, let us fix some notation: First, let $\Val_k^\infty\subset\Val^\infty$ denote the subspace of $k$-homogeneous elements, and let $\mu_1\in\Val_1^\infty$ be the first intrinsic volume. Further, assume $0\leq k\leq \lfloor\frac n2\rfloor$, where $n$ is the dimension of the underlying Euclidean space, and let us call a valuation $\phi\in\Val_k^\infty$ {\it primitive} if $\phi\cdot\mu_1^{n-2k+1}=0$. Finally, consider the {\it Alesker-Hodge-Riemann pairing} $Q\maps{\Val_k^\infty\times\Val_k^\infty}{\Val_n^\infty}$ given by
\begin{align}
Q(\phi,\psi)=\b\phi\cdot\psi\cdot\mu_1^{n-2k}.
\end{align}
After the standard identification $\Val_n^\infty\cong\CC$, $Q$ becomes an (a priori non-degenerate) Hermitian form. As far as its signature on primitive elements is concerned, a careful look at the various known subalgebras of $G$-invariant valuations---in particular, at the sufficiently complicated case $G=\Spin(9)$, resolved recently in the author's Ph.D. thesis \cite{KotrbatyPhD}---shows that it seems to precisely reflect the parity of the degree $k$. Since all such valuations are in fact {\it even} (curiously, even when $-\id\notin G$), it is natural to expect a more general phenomenon underlying the behavior, at least in this case. Indeed, it is our first main result that
\begin{itheorem}
\label{thm:HReven}
Let $0\leq k\leq \lfloor\frac n2\rfloor$. Any non-zero primitive even valuation $\phi\in\Val_k^\infty$ satisfies
\begin{align}
\label{eq:HReven}
(-1)^k\,Q(\phi,\phi)>0.
\end{align}
\end{itheorem}

The general situation turns out to be slightly more subtle: Perhaps surprisingly, already in the case $k=1$, which is clearly the first non-trivial one to deal with as $0$-homogeneous valuations are constant and thus even, the dependence of the signature of the Alesker-Hodge-Riemann pairing on the parity of a valuation starts manifesting itself. Specifically, let even valuations have parity 0, and odd valuations parity 1. The second main result of our article is then as follows:
\begin{itheorem}
\label{thm:HR1}
Let $s\in\{0,1\}$. Any non-zero primitive valuation $\phi\in\Val_1^\infty$ of parity $s$ satisfies
\begin{align}
\label{eq:HR1}
(-1)^{1+s}\,Q(\phi,\phi)>0.
\end{align}
\end{itheorem}

The proofs of Theorems \ref{thm:HReven} and \ref{thm:HR1} both rely on representing the valuations in question by smooth functions (on Grassmannians in the former case, on the sphere in the latter). In either situation the Lefschetz map boils down to a certain $\SO(n)$-equivariant integral transform which makes it possible to employ representation theory and to deduce the result from the sign of its eigenvalues. In the case of 1-homogeneous valuations, a large part of the work has already been done by Bernig and Hug in their recent article \cite{BernigHug18}, upon which our proof is heavily based.

\subsection{Conjectures}

Although the techniques behind the two main results of the present article certainly reach their limits here and, therefore, different arguments would be needed in order to prove a stronger statement,  we believe it is reasonable to expect the Hodge-Riemann relations to hold in a more general context. Let us formulate precise conjectures. First, the natural generalization of Theorems \ref{thm:HReven} and \ref{thm:HR1} would certainly be as follows:
\begin{iconjecture}
\label{con:HR}
Let $0\leq k\leq \lfloor\frac n2\rfloor$ and $s\in\{0,1\}$. Any non-zero primitive valuation $\phi\in\Val_k^\infty$ of parity $s$ satisfies
\begin{align}
\label{eq:HR}
(-1)^{k+s}\,Q(\phi,\phi)>0.
\end{align}
\end{iconjecture}

Second, applying the Alesker-Fourier transform, the Hodge-Riemann-type relations \eqref{eq:HR} can be reformulated in terms of the Bernig-Fu convolution. In this case, the algebra is graded by the co-degree, and the Lefschetz map is induced by the penultimate intrinsic volume $\mu_{n-1}\in\Val_{n-1}^\infty$. Let us say that a valuation $\phi\in\Val_{n-k}^\infty$ is {\it co-primitive} if $\phi*\mu_{n-1}^{*(n-2k+1)}=0$, and consider the {\it Bernig-Fu-Hodge-Riemann pairing} $\wt Q\maps{\Val_{n-k}^\infty\times\Val_{n-k}^\infty}{\Val_0^\infty\cong\CC}$ given by
\begin{align}
\wt Q(\phi,\psi)=\b\phi*\psi*\mu_{n-1}^{*(n-2k)}.
\end{align}
We show that Conjecture \ref{con:HR} is equivalent to
\begin{iconjecture}
\label{con:HRc}
Let $0\leq k\leq \lfloor\frac n2\rfloor$. Any non-zero co-primitive valuation $\phi\in\Val_{n-k}^\infty$ satisfies
\begin{align}
\label{eq:HRc}
(-1)^{k}\,\wt Q(\phi,\phi)>0.
\end{align}
\end{iconjecture}
Notice that this setting appears to be the more natural one as \eqref{eq:HRc} no longer depends on the parity; we shall also see that it is more appropriate for applications and further extensions. Further, let us point out that Conjecture \ref{con:HRc} is true in the proven cases of Conjecture \ref{con:HR}, i.e., for even valuations or if $k=1$. 

Very often the Kähler package is as powerful as proclaimed only if it comes in a more general version than we have discussed here. Namely, instead of being the composition of several copies of the same map, the Lefschetz map (in both HL and HR) is considered to be composed of general elements of a certain cone of operators. For Kähler manifolds, in fact, such {\it mixed} versions of HL and HR were proved only recently by Dinh and Nguy\^en \cite{DinhNguyen06} who completed earlier work of Gromov \cite{Gromov90} and Timorin \cite{Timorin98}. Motivated by these developments, we, finally, conjecture the mixed versions of the hard Lefschetz theorem and of the Hodge-Riemann relations for the algebra $\Val^\infty$ as follows:
\begin{iconjecture}
\label{con:mixed}
Let $0\leq k\leq \lfloor\frac n2\rfloor$ and let $K_{2k},\dots,K_n$ be convex bodies with smooth boundary and positive curvature.
\begin{enuma}
\item The mapping $\Val^\infty_{n-k}\rightarrow\Val_k^\infty$ given by
\begin{align}
\phi\mapsto\phi*V(\Cdot[n-1],K_{2k+1})*\cdots* V(\Cdot[n-1],K_n)
\end{align}
is an isomorphism. 
\item Let $\phi\in\Val_{n-k}^\infty$ be non-zero and such that
\begin{align}
\phi*V(\Cdot[n-1],K_{2k})*V(\Cdot[n-1],K_{2k+1})*\cdots* V(\Cdot[n-1],K_n)=0.
\end{align}
Then
\begin{align}
\label{eq:mixedHR}
(-1)^k\,\b\phi*\phi*V(\Cdot[n-1],K_{2k+1})*\cdots *V(\Cdot[n-1],K_n)>0.
\end{align}
\end{enuma}
\end{iconjecture}

Observe that the two parts of Conjecture \ref{con:mixed} generalize their non-mixed counterparts Theorem \ref{thm:HLconvolution} below and Conjecture \ref{con:HRc}, respectively, since $\mu_{n-1}$ is positively proportional to the mixed volume $V(\Cdot[n-1],D)$, with $D$ being the Euclidean ball. Let us also point out that reformulating the previous conjecture back in terms of the Alesker product would involve---at least under the additional assumption of central symmetry---the so-called {\it Holmes-Thompson intrinsic volumes}, i.e., valuations naturally assigned to a general smooth Minkowski space (cf. \cite{FaifmanWannerer21,Bernig07,AlvarezPaivaBerck06}).

\subsection{Applications to inequalities of geometric type}

It turns out that some of the purely algebraic formulations of the aforedescribed results and conjectures can in fact acquire a geometric meaning. This finally provides the anticipated connection between the algebraic theory of valuations and the isoperimetric inequalities.

First, we show that Conjecture \ref{con:HRc}---used in the situation $k=1$ where we know it is true---yields an important special case of the Aleksandrov-Fenchel inequality, namely, Minkowski's second inequality for one body being the Euclidean ball, or, in other words, the isoperimetric inequality between the first and the second intrinsic volume. This follows a simple argument, a special case of which was communicated to us by S. Alesker. 

Furthermore, using a `mixed' version of the same argument, we show that the corresponding special case of the conjectured mixed Hodge-Riemann relations (Conjecture \ref{con:mixed})---if true---yields the Aleksandrov-Fenchel inequality in its full generality. In particular, this would establish also the rest of the isoperimetric inequalities. Recall that one way to formulate these is to say that the sequence of (properly normalized) intrinsic volumes is log-concave. To quote June Huh \cite{Huh18}, {\it the log-concavity of a sequence is not only important because of its applications but because it hints the existence of a structure that satisfies PD, HL, and HR}. Conjecture \ref{con:mixed} may thus be viewed as yet another argument in favor of this phenomenon, and vice versa. 

To conclude the introduction, let us point out that among the numerous known proofs of the Aleksandrov-Fenchel inequality \cite{Aleksandrov37,Aleksandrov38,McMullen93,Wang18,ShenfeldvanHandel19,Cordero-ErausquinEA19} (see also \cite{BuragoZalgaller1988}, \S27), one is particularly relevant to our approach; namely, McMullen \cite{McMullen93} regards the inequality as a special case of the (mixed) Hodge-Riemann relations in his polytope algebra \cite{McMullen89}. The interplay between McMullen's algebra and the algebra $\Val^\infty$ was recently investigated by Bernig and Faifman \cite{BernigFaifman16} who showed that the former is a subalgebra of the (partial) convolution algebra $\Val^{-\infty}$ of {\it generalized} translation-invariant valuations in which $\Val^\infty$ is densely embedded. From this point of view it would certainly be interesting to know whether an explicit relationship exists  between Conjecture \ref{con:mixed} and McMullen's result.

\subsection*{Acknowledgements}

I am grateful to Karim Adiprasito and Fr\'ed\'eric Chapoton for extremely valuable discussions out of which this project actually arose; to Semyon Alesker for his constant interest in this work and for numerous helpful suggestions, in particular, for the beautiful idea of applying the results to geometric-type inequalities; to the DAAD Foundation that supported my stay at the Tel Aviv University where this work was initiated; as well as to Franz Schuster and Thomas Wannerer for their encouragement and many useful comments on earlier versions of the paper. Also, I wish to thank to an anonymous referee for useful remarks.

\section{Preliminaries}

\subsection{Valuations on convex bodies}

\label{ss:valuations}

Let us begin by recalling some necessary facts from both classical and modern theory of valuations. Our references for this and the following section are monographs \cite{Alesker2018} and \cite{Schneider2014}.

To establish notation, let $\K$ be the family of convex bodies, i.e., non-empty compact convex sets in $\RR^n$. On this space, the operations of Minkowski addition and scalar multiplication are, respectively, given for $K,L\in\K$ and $\alpha\in\RR$ as follows:
\begin{align}
K+L&=\{x+y\mid x\in K,y\in L\},\\
\alpha K&=\{\alpha x\mid x\in K\}.
\end{align}

A functional $\phi\maps{\K}{\CC}$ is called a {\it valuation} if
\begin{align}
\phi(K\cup L)+\phi(K\cap L)=\phi(K)+\phi(L)
\end{align}
holds for any $K,L\in\K$ with $K\cup L\in\K$. A valuation $\phi$ is said to be {\it translation invariant} if $\phi(K+\{x\})=\phi(K)$
holds for any $K\in\K$ and $x\in\RR^n$, and {\it continuous} if it is so with respect to the Hausdorff metric $
\rho_H(K,L)=\inf\{\e>0\mid K\subset L+\e D,L\subset K+\e D\}$, where $D\subset\RR^n$ is the (origin-centered) unit Euclidean ball. We shall consider entirely valuations enjoying both these properties and denote the vector space of all such by $\Val$. Let further $\Val_k\subset\Val$ be the subspace of valuations satisfying $\phi(\lambda K)=\lambda^k\phi(K)$ for any $\lambda>0$ and $K\in\K$. Remarkably, one has the McMullen grading
\begin{align}
\label{eq:McMullenGr}
\Val=\bigoplus_{k=0}^n\Val_k,
\end{align}
in consequence of which $\Val$ becomes a Banach space with respect to $\Norm\phi=\sup_{K\subset D}\norm{\phi(K)}$. Another, obvious, grading is with respect to parity:
\begin{align}
\label{eq:parityGr}
\Val=\Val^0\oplus\Val^1,
\end{align}
where $\Val^s=\{\phi\in\Val\mid\phi(-K)=(-1)^s\phi(K)\text{ for any }K\in\K\}$. We denote $\Val^s_k=\Val_k\cap\Val^s$. There is a natural $\GL(n,\RR)$ action on $\Val$ given by $(g,\phi)\mapsto\phi\circ g^{-1}$. Clearly, each $\Val_k^s$ is then a closed invariant subspace. The following result is central to the theory of valuations:
\begin{theorem}[Alesker \cite{Alesker01}]
\label{thm:IT}
For any $0\leq k\leq n$ and $s=0,1$, the $\GL(n,\RR)$ module $\Val_k^s$ is irreducible, i.e., admits no proper closed invariant subspace.
\end{theorem}
Observe that the statement is only non-trivial for $0<k<n$. Indeed, one has $\Val_0=\linspan\{\chi\}$, where $\chi\equiv1$ is the Euler characteristic, and similarly $\Val_n=\linspan\{\vol_n\}$ is spanned by the Lebesgue measure. The former is easy to see, while the latter is a deep result of Hadwiger. As it is usual, let us identify $\Val_0$ and $\Val_n$ with $\CC$ via $\chi$ and $\vol_n$, respectively. 

A broad generalization of the two prime examples is provided by the concept of {\it mixed volumes}. Recall that the mixed volume of $K_1,\dots,K_n\in\K$ is defined as 
\begin{align}
\label{eq:MV1}
V(K_1,\dots,K_n)=\frac1{n!}\frac{\partial^n}{\partial\lambda_1\cdots\partial\lambda_n}\bigg|_{\lambda_1,\dots,\lambda_n=0}\vol_n(\lambda_1K_1+\cdots+\lambda_n K_n).
\end{align}
According to a classical result of Minkowski, the function being differentiated on the right-hand side of \eqref{eq:MV1} is actually a homogeneous polynomial of degree $n$. It is well known that $V$ is totally symmetric, real valued, and non-negative, and that $V(\Cdot[k],K_{k+1},\dots,K_n)\in\Val_k$ for any $K_i$, where $[k]$ stands for $k$ copies. In particular,
\begin{align}
\label{eq:IV}
\mu_k=c_{k}V(\Cdot[k],D[n-k])
\end{align}
is the $k$-th {\it intrinsic volume}. The exact values of the normalizing constants $c_k>0$ (which can be found e.g. in \cite{KlainRota1997}, p. 141) will not be important for what follows; let us only mention that they are chosen such that $\mu_0=\chi$ and $\mu_n=V(\Cdot[n])=\vol_n$. 

Famously, the intrinsic volumes satisfy, for any $K\in\K$, the {\it Aleksandrov-} or {\it isoperimetric inequalities}
\begin{align}
\label{eq:II}
\frac{\mu_1(K)}{\mu_1(D)}\geq\left(\frac{\mu_2(K)}{\mu_2(D)}\right)^{\frac 12}\geq\cdots\geq\left(\frac{\mu_k(K)}{\mu_k(D)}\right)^{\frac 1k}\geq\cdots\geq\left(\frac{\mu_n(K)}{\mu_n(D)}\right)^{\frac 1n}.
\end{align}
In fact, \eqref{eq:II} is only a consequence of a much stronger result, namely, the {\it Aleksandrov-Fenchel inequality} for mixed volumes: For all $K_1,\dots,K_n\in\K$, one has
\begin{align}
\label{eq:AF}
V(K_1,K_2,K_3,\dots,K_n)^2\geq V(K_1,K_1,K_3,\dots,K_n)\,V(K_2,K_2,K_3,\dots,K_n).
\end{align}

According to McMullen's conjecture \cite{McMullen80}, whose proof is an easy consequence of Theorem \ref{thm:IT}, mixed volumes span a dense subset of $\Val$. By the standard polarization formulas, the same is true for valuations of the form $\vol_n(\Cdot+K)$, with $K\in\K$.

\subsection{Smooth valuations}

\label{ss:smooth}

A valuation $\phi\in\Val$ is said to be {\it smooth} if the Banach-space-valued mapping $\GL(n,\RR)\rightarrow\Val$ given by $g\mapsto \phi\circ g^{-1}$ is infinitely differentiable. It is a general fact from representation theory that the subspace $\Val^\infty\subset\Val$ of such valuations is invariant, dense, and carries a natural Fr\'echet-space topology (stronger than that induced from $\Val$), with which $\Val^\infty$ will be tacitly assumed to be endowed. In this connection, observe that the irreducibility theorem (Theorem \ref{thm:IT}) holds verbatim for the $\GL(n,\RR)$ modules $\Val^{s,\infty}_k=\Val^s_k\cap\Val^\infty$ as well as the gradings \eqref{eq:McMullenGr} and \eqref{eq:parityGr} for, respectively, $\Val^\infty_k=\Val_k\cap\Val^\infty$ and $\Val^{s,\infty}=\Val^s\cap\Val^\infty$.

The notion of smoothness is essential to the modern valuation theory as $\Val^\infty$ can be naturally equipped with an array of striking algebraic structures that do not, however, extend to $\Val$. To this end, let us recall that the valuations $V(\Cdot[k],K_{k+1},\dots,K_n)$ and $\vol_n(\Cdot+K)$ are smooth provided $K_i$ and $K$, respectively, belong to the class $\K_+^\infty$ of convex bodies with non-empty interior, smooth boundary, and positive Gauss-Kronecker curvature. In particular, the intrinsic volumes are such. Against this background, the {\it Alesker product} and the {\it Bernig-Fu convolution} are, respectively, given as follows:
\begin{theorem}[Alesker \cite{Alesker04}]
\label{thm:AP}
Let $\Delta\maps{\RR^n}{\RR^{2n}}\mape{x}{(x,x)}$ be the diagonal embedding. There exists a unique bilinear, continuous product $\cdot$ on $\Val^\infty$ such that
\begin{align}
\label{eq:product1}
\vol_n(\Cdot+K)\cdot\vol_n(\Cdot+L)=\vol_{2n}(\Delta(\Cdot)+K\times L)
\end{align}
holds for any $K,L\in\K_+^\infty$. Moreover,
\begin{enuma}
\item $\cdot$ is commutative and associative,
\item $\phi\cdot\chi=\phi$ for any $\phi\in\Val^\infty$,
\item $\Val^{s,\infty}_k\cdot\Val^{r,\infty}_l\subset\Val_{k+l}^{q,\infty}$ with $q=r+s\mod 2$,
\item $\mu_k\cdot\mu_l=c_{k,l}\mu_{k+l}$ for some $c_{k,l}>0$,
\item the product pairing $\Val_k^\infty\times\Val^\infty_{n-k}\rightarrow\Val_n^\infty\cong\CC$ is non-degenerate, i.e., for any non-zero $\phi\in\Val_k^\infty$ there is $\psi\in\Val_{n-k}^\infty$ with $\phi\cdot\psi\neq0$.
\end{enuma}
\end{theorem}

\begin{theorem}[Bernig, Fu \cite{BernigFu06}]
There exists a unique bilinear, continuous product $*$ on $\Val^\infty$ such that
\begin{align}
\label{eq:convDef}
\vol_n(\Cdot+K)*\vol_n(\Cdot+L)=\vol_n(\Cdot+K+L)
\end{align}
holds for any $K,L\in\K_+^\infty$.
\end{theorem}

Intertwining the product and the convolution is the so-called {\it Alesker-Fourier transform}:

\begin{theorem}[Alesker \cite{Alesker11}]
\label{thm:AFT}
There exists a canonical isomorphism $\FF\maps{\Val^\infty}{\Val^\infty}$ of Fr\'echet spaces such that
\begin{enuma}
\item $\FF(\phi\cdot\psi)=\FF\phi*\FF\psi$ holds for any $\phi,\psi\in\Val^\infty$,
\item $\FF\Val_{k}^{s,\infty}=\Val_{n-k}^{s,\infty}$,
\item $\FF^2=(-1)^s\id$ on $\Val^{s,\infty}$,
\item $\FF\mu_k=\mu_{n-k}$.
\end{enuma}
\end{theorem}

Synthesis of Theorems \ref{thm:AP} and \ref{thm:AFT} yields at once
\begin{corollary}
\label{cor:BFC}
\quad
\begin{enuma}
\item $*$ is commutative and associative,
\item $\phi*\vol_n=\phi$ for any $\phi\in\Val^\infty$,
\item $\Val^{s,\infty}_{n-k}*\Val^{r,\infty}_{n-l}\subset\Val_{n-k-l}^{q,\infty}$ with $q=r+s\mod 2$,
\item $\mu_{n-k}*\mu_{n-l}=c_{k,l}\mu_{n-k-l}$ for some $c_{k,l}>0$,
\item for any non-zero $\phi\in\Val_k^\infty$ there is $\psi\in\Val_{n-k}^\infty$ with $\phi*\psi\neq0$.
\end{enuma}
\end{corollary}

In fact, the items (a)--(d) of Corollary \ref{cor:BFC} are an easy consequence of \eqref{eq:convDef}; see also Lemma \ref{lem:conv} below. Items (e) of Theorem \ref{thm:AP} and Corollary \ref{cor:BFC} are versions of {\it Poincar\'e duality} for valuations. The following properties of the convolution will be also useful for us:
\begin{lemma}[Bernig, Fu \cite{BernigFu06}]
\label{lem:conv}
\quad
\begin{enuma}
\item Let $k+l\leq n$. There is a constant $c_{k,l}>0$ such that for any $K_1,\dots,K_k,L_1,\dots,L_l\in\K_+^\infty$,
\begin{align}
\label{eq:MVconv}
\begin{split}
&V(\Cdot[n-k],K_1,\dots,K_k)*V(\Cdot[n-l],L_1,\dots,L_l)\\
&\quad=c_{k,l}V(\Cdot[n-k-l],K_1,\dots,K_k,L_1,\dots,L_l).
\end{split}
\end{align}
\item For any $\phi\in\Val^\infty$ and $K\in\K$, one has
\begin{align}
\label{eq:convLefschetz}
(\mu_{n-1}*\phi)(K)=\frac12\frac{d}{d\lambda}\bigg|_{\lambda=0}\phi(K+\lambda D).
\end{align}
\end{enuma}
\end{lemma}

Finally, as discussed in the introduction, two versions of the {\it hard Lefschetz theorem} hold for smooth valuations:
\begin{theorem}[Bernig, Bröcker \cite{BernigBroecker07}]
\label{thm:HLconvolution}
Let $0\leq k\leq \lfloor\frac n2\rfloor$. The mapping $\Val_{n-k}^\infty\rightarrow\Val_{k}^\infty$ given by
\begin{align}
\phi\mapsto\phi*\mu_{n-1}^{*(n-2k)}
\end{align}
is an isomorphism of Fr\'echet spaces.
\end{theorem}
\begin{corollary}
\label{cor:HLproduct}
Let $0\leq k\leq \lfloor\frac n2\rfloor$. The mapping $\Val_{k}^\infty\rightarrow\Val_{n-k}^\infty$ given by
\begin{align}
\phi\mapsto\phi\cdot\mu_1^{n-2k}
\end{align}
is an isomorphism of Fr\'echet spaces.
\end{corollary}

Observe that it follows at once from \eqref{eq:product1} that
\begin{align}
\label{eq:APequivariance}
(\phi\circ g^{-1})\cdot(\psi\circ g^{-1})=(\phi\cdot\psi)\circ g^{-1}
\end{align}
holds for any $\phi,\psi\in\Val^\infty$ and $g\in\GL(n,\RR)$. In particular, the product with $\mu_1$ commutes with the maximal compact subgroup $\O(n)$ of $\GL(n,\RR)$ as the intrinsic volumes are $\O(n)$ invariant.

\subsection{Functions on Grassmannians and integral transforms}

\label{ss:fGr}

Let us now collect some known facts about the Hilbert space of square-integrable functions on the Grassmann manifold $\Grass_k$ of $k$-dimensional real subspaces in $\RR^n$. Enough for us will be to assume $0\leq k\leq \lfloor\frac n2\rfloor$. 

First of all, $L^2(\Grass_k)$ has an obvious $\SO(n)$-module structure. In this connection, recall that the family of irreducible $\SO(n)$ representations is parametrized by the set $\Lambda$ of their highest weights, where, if $n=2m+1$ is odd,
\begin{align}
\Lambda=\left\{(\lambda_1,\dots,\lambda_m)\in\ZZ^m\mid \lambda_1\geq\lambda_2\geq\cdots\geq\lambda_m\geq0\right\},
\end{align}
while for $n=2m$ even,
\begin{align}
\Lambda=\left\{(\lambda_1,\dots,\lambda_m)\in\ZZ^m\mid \lambda_1\geq\lambda_2\geq\cdots\geq\lambda_{m-1}\geq\norm{\lambda_m}\right\}.
\end{align}
In both cases, for $0\leq k\leq m=\lfloor \frac n2\rfloor$, let us denote
\begin{align}
\Lambda_k^0=\Lambda\cap\left((2\ZZ)^k\times\{0\}^{m-k}\right).
\end{align}
It is well known from works of Sugiura \cite{Sugiura62} and Takeuchi \cite{Takeuchi73} (see also \cite{Strichartz75}) that the decomposition of $L^2(\Grass_k)$ into irreducible $\SO(n)$ modules is multiplicity free. More precisely,
\begin{align}
\label{eq:CGrassk}
L^2(\Grass_k)=\widehat{\bigoplus_{\lambda\in\Lambda_k^0}}\calH_\lambda,
\end{align}
where $\calH_\lambda\subset C^{\infty}(\Grass_k)$ is the irreducible $\SO(n)$-module corresponding to the highest weight $\lambda$, and $\widehat\bigoplus$ denotes the ($L^2$-orthogonal) Hilbert-space direct sum (cf. \cite{Sepanski2007}, \S3.2.2).

For each $\lambda=(2m_1,\dots,2m_k,0,\dots,0)\in\Lambda_k^0$, the highest-weight vector $h_\lambda\in\calH_\lambda$ (which is unique up to scaling) was described by Strichartz \cite{Strichartz75}: Assume first $k\geq 1$. Take an arbitrary $E\in\Grass_k$ and let
\begin{align*}
\begin{pmatrix}x_{1,1}\\\vdots\\x_{n,1}\end{pmatrix},\dots,\begin{pmatrix}x_{1,k}\\\vdots\\x_{n,k}\end{pmatrix}
\end{align*}
be an orthonormal basis of $E$. For $1\leq l\leq k$, we first define
\begin{align}
\label{eq:A[l]}
A[l]=\begin{pmatrix}x_{1,1}+\sqrt{-1}x_{2,1}&\cdots&x_{1,k}+\sqrt{-1}x_{2,k}\\
x_{3,1}+\sqrt{-1}x_{4,1}&\cdots&x_{3,k}+\sqrt{-1}x_{4,k}\\
\vdots&&\vdots\\
x_{2l-1,1}+\sqrt{-1}x_{2l,1}&\cdots&x_{2l-1,k}+\sqrt{-1}x_{2l,k}\\
\end{pmatrix}\in\CC^{l\times k}.
\end{align}
Then, if $m_k\geq0$, one puts
\begin{align}
h_\lambda(E)=\left(\prod_{l=1}^{k-1}\det\left(A[l]A[l]^t\right)^{m_l-m_{l+1}}\right)\det\left(A[k]A[k]^t\right)^{m_k},
\end{align}
and, if $m_k<0$,
\begin{align}
h_\lambda(E)=\left(\prod_{l=1}^{k-1}\det\left(A[l]A[l]^t\right)^{m_l-\norm{m_{l+1}}}\right)\b{\det\left(A[k]A[k]^t\right)}^{\norm{m_k}}.
\end{align}
It is straightforward to verify that this definition is independent of the choice of the orthonormal basis. If $k=0$, one necessarily has $\lambda=(0,\dots,0)$ and defines $h_\lambda(\{0\})=1$.

Let us recall two important integral transforms on Grassmanians. First, for $0\leq k\leq l\leq n$, we define the {\it Radon transform} $R_{k,l}\maps{L^2(\Grass_k)}{L^2(\Grass_l)}$ as follows:
\begin{align}
\label{eq:Rkl}
(R_{k,l}f)(E)=\int_{\Grass_k(E)}f(F)\,d F,
\end{align}
where $\Grass_k(E)=\{F\in\Grass_k\mid F\subset E\}$ is equipped with the Haar probability measure $dF$. Notice that $R_{k,k}=\id$. Reversing the inclusion $F\subset E$, one can extend the definition of the Radon transform also to $k>l$. For our purpose, however, the above form will be sufficient.

Second, consider $E,F\in\Grass_k$, $0\leq k\leq n$. Take any compact subset $A\subset E$ with $\vol_k(A)\neq0$ and put
\begin{align}
\label{eq:cosEF}
\norm{\cos(E,F)}=\frac{\vol_k(\pi_FA)}{\vol_k(A)},
\end{align}
where $\pi_F\maps{\RR^n}{F}$ is the orthogonal projection. Notice that \eqref{eq:cosEF} is independent of the choice of the set $A$. Then the {\it cosine transform} $T_k\maps{L^2(\Grass_k)}{L^2(\Grass_k)}$ is defined as
\begin{align}
(T_kf)(E)=\int_{\Grass_k}\norm{\cos(E,F)}f(F)\,d F,
\end{align}
where $dF$ is the Haar measure. Observe that both $R_{k,l}$ and $T_k$ are linear, continuous, and $\SO(n)$ equivariant, in consequence of which they map smooth functions to smooth functions. In general, one can define the cosine transform between functions on Grassmanians of distinct rank, or one can consider the so-called $\alpha$-cosine transform where the kernel is raised to $\alpha\in\CC$. Again, such generalizations will not be considered here. What we shall need instead are known eigenvalues of the cosine transform. For $\nu\in\CC$ and $k\in\NN_0$, consider the Pochhammer symbol
\begin{align}
(\nu)_k=\prod_{j=0}^{k-1}(\nu+j).
\end{align}
\begin{lemma}[Zhang \cite{Zhang09}, see also \cite{AleskerGS16}]
\label{lem:multipT}
Let $1\leq k\leq \lfloor \frac n2\rfloor$ and $\lambda=(2m_1,\dots,2m_k,0,\dots,0)\in\Lambda_k^0$. Then, for some normalizing constant $c_k>0$,
\begin{align}
T_k|_{\calH_\lambda}=c_k\,\prod_{j=1}^k\frac{\left(1+\frac j2-\norm{m_j}\right)_{\norm{m_j}}}{\left(1+\frac n2-\frac j2\right)_{\norm{m_j}}}\, \id|_{\calH_\lambda}.
\end{align}
\end{lemma}

\subsection{Functional representations of even valuations}

\label{ss:even}

There are two important ways to represent {\it even} smooth valuations in terms of smooth functions on Grassmanians. First, the {\it Crofton map} $\Cr\maps{C^\infty(\Grass_k)}{\Val_k^{0,\infty}}$ is given by
\begin{align}
\label{eq:Cr}
\Cr (f)(K)=\int_{\Grass_k}f(E)\vol_k(\pi_E K) \,d E,\quad f\in C^\infty(\Grass_k),K\in\calK.
\end{align}
It is readily verified that the map is indeed well defined and linear (cf. \cite{AleskerBernstein04}). Importantly,
\begin{theorem}[Alesker, Bernstein \cite{AleskerBernstein04}, see also \cite{Alesker03hard}]
\label{thm:AB}
Let $0\leq k\leq \lfloor \frac n2\rfloor$. The following restriction of the Crofton map is an isomorphism:
\begin{align}
\label{eq:Crsmooth}
\Cr\maps{C^\infty(\Grass_k)\cap\widehat{\bigoplus_{\substack{\lambda\in\Lambda_k^0\\ \norm{\lambda_2}\leq 2}}}\calH_\lambda}{\Val^{0,\infty}_k}.
\end{align}
\end{theorem}

Let us point out that since the isomorphism commutes with the natural action of $\SO(n)$, \eqref{eq:Crsmooth} also gives the decomposition of $\Val^{0,\infty}_k$ into irreducible $\SO(n)$ modules. In the sense of the preceding theorem, we shall always denote 
\begin{align}
f_\phi=\Cr^{-1}(\phi),\quad \phi\in\Val^{0,\infty}_k.
\end{align}
In other words, each $\phi\in\Val^{0,\infty}_k$ will be represented uniquely by means of
\begin{align}
\phi(K)=\int_{\Grass_k}f_\phi(E)\vol_k(\pi_E K)\,d E,\quad K\in\K.
\end{align}

Second, the {\it Klain map} $\Klain\maps{\Val_k^{0,\infty}}{C^\infty(\Grass_k)}$ is defined as follows: According to Hadwiger's characterization theorem, the restriction of any $\phi\in\Val_k^{0,\infty}$ to an arbitrary $k$-plane $E\in\Grass_k$ must be a multiple of the Lebesgue measure $\vol_k$ on $E$. Then, one defines $\Klain(\phi)=\Klain_\phi$ by
\begin{align}
\phi|_E=\Klain_\phi(E)\,\vol_k.
\end{align}
It can be shown that this is a well-defined, linear, and $\SO(n)$-equivariant mapping. Importantly, it was shown by Klain \cite{Klain00} that $\Klain$ is injective. As a side note, one has $\Klain\circ\Cr=T_k$ on $C^{\infty}(\Grass_k)$.

We shall need a few more results concerning the expression of some of the algebraic structures introduced in \S\ref{ss:smooth} (restricted to even smooth valuations) in terms of the Crofton and Klain map. First, the Alesker-Fourier transform of $\phi\in\Val_k^{0,\infty}$ admits a simple description in terms of the operation $\perp\maps{\Grass_{n-k}}{\Grass_k}$ of taking the orthogonal complement as follows:
\begin{align}
\label{eq:FFperp}
f_{\FF\phi}=\perp^* f_{\phi}
\end{align}
(see \cite{Alesker03hard,Alesker11}). Second, the Lefschetz map boils down to
\begin{lemma}[Alesker \cite{Alesker04hard}]
Let $0\leq k+l\leq n$. There is $c_{k,l}>0$ such that for any $\phi\in\Val_k^{0,\infty}$,
\begin{align}
\label{eq:Klaintimes}
\Klain_{\phi\cdot \mu_l}=c_{k,l}\,T_{k+l}\circ R_{k,k+l}(f_\phi).
\end{align}
\end{lemma}
Let us point out that the constant $c_{k,l}$ was only specified to be {\it non-zero} in the original version of the previous lemma as formulated in \cite{Alesker04hard}. However, it is easy to trace back through Alesker's proof to see that $c_{k,l}$ is indeed {\it real} and {\it positive}. Finally, as for the Alesker-Poincar\'e pairing,
\begin{lemma}[Bernig, Fu \cite{BernigFu06}]
For $\varphi,\psi\in\Val_k^{0,\infty}$ one has
\label{lem:pd+}
\begin{align}
\label{eq:pd+}
\FF\varphi\cdot\psi=\int_{\Grass_k}f_\varphi(E)\,\Klain_\psi(E)\,d E.
\end{align}
\end{lemma}
Recall that $\Val_n^\infty$ is identified with $\CC$ as usual, i.e., via $\vol_n$.

\subsection{Spherical valuations}

\label{ss:spherical}

There is another broad class of smooth valuations represented faithfully by means of functions, this time on the unit sphere. Recall that
\begin{align}
\label{eq:L2S}
L^2(S^{n-1})=\widehat{\bigoplus_{q\in\NN_0}}\calS_q,
\end{align}
where $\calS_q\subset C^\infty(S^{n-1})$ is the space of spherical harmonics, i.e. restrictions of $q$-homogeneous harmonic polynomials on $\RR^n$. If $f\in L^2(S^{n-1})$ is smooth, then the expansion \eqref{eq:L2S} converges also with respect to the standard Fr\'echet-space structure on $C^\infty(S^{n-1})$ (see \cite{Morimoto1998}, \S2.6).

For $0\leq k\leq n-1$, let further $S_k(K,\Cdot)$ be the $k$-th area measure of a convex body $K\in\K$ (we refer to \S4.2 of \cite{Schneider2014} for its construction). It is well known that the area measures satisfy a Steiner-type formula: For $1\leq k\leq n-1$, there are constants $c_{k,j}>0$ such that
\begin{align}
\label{eq:SteinerSk}
S_k(K+\lambda D,\omega)=\sum_{j=0}^k c_{k,j} \lambda^jS_{k-j}(K,\omega)
\end{align}
holds for any $\lambda>0$, $K\in\calK$, and any Borel subset $\omega\subset S^{n-1}$. It is easy to see that for a smooth function $f\in C^\infty(S^{n-1})$, the valuation defined by
\begin{align}
\label{eq:spherical}
\mu_{k,f}(K)=\int_{S^{n-1}}f(y)dS_k(K,y), \quad K\in\K,
\end{align}
belongs to $\Val^\infty_k$. Valuations of the form \eqref{eq:spherical} are called {\it spherical}. Observe that, in particular, the intrinsic volumes are such, as $\mu_{k,1}$ is proportional to $\mu_k$. In this connection, let us take the liberty of assuming that the area measures are normalized such that
\begin{align}
\label{eq:muk1}
\mu_{k,1}=\mu_k.
\end{align}
Notice also that $\mu_{k,f}=0$ provided $f\in\calS_1$ is the restriction of a linear functional, and that $\mu_{k,f}$ is even/odd if the same is true for $f$.

The following recent results concerning spherical valuations will be useful for us:

\begin{lemma}[Bernig, Hug \cite{BernigHug18}]
\label{lem:spherical}
Let $1\leq k\leq n-1$ and $q,r\in\NN_0$ with $q,r\neq1$.
\begin{enuma}
\item There is a constant $c_{k,q}>0$ such that for any $f\in\calS_q$,
\begin{align}
\label{eq:FFmukf}
\FF\mu_{k,f}=c_{k,q}(\sqrt{-1})^q\mu_{n-k,f}.
\end{align}
\item There is a constant $\tilde c_{k,q}>0$ such that for any $f\in\calS_q$ and $g\in\calS_r$,
\begin{align}
\label{eq:pdmufg}
\mu_{k,f}\cdot\mu_{n-k,g}=\begin{cases}\tilde c_{k,q}(-1)^{q}\left(1-\frac{q(n+q-2)}{n-1}\right)\int_{S^{n-1}}f(y)\,g(y)\,dy&\text{if }q=r,\\ 0&\text{otherwise,}\end{cases}
\end{align}
where $\Val^\infty_n\cong\CC$ as usual.
\end{enuma}
\end{lemma}

\begin{lemma}[Alesker \cite{BergEtAl18}]
\label{lem:spherical1}
Let $C_0^\infty(S^{n-1})\subset C^\infty(S^{n-1})$ be the Fr\'echet subspace of functions whose orthogonal projection to $\calS_1$ is trivial. Then the mapping $C_0^\infty(S^{n-1})\rightarrow\Val_1^\infty$ given by
\begin{align}
f\mapsto\mu_{1,f}
\end{align}
is an isomorphism of Fr\'echet spaces.
\end{lemma}

\section{Primitive valuations, the Lefschetz decomposition for valuations,\\and the Alesker-Hodge-Riemann pairing}

Let $0\leq k\leq \lfloor \frac n2\rfloor$. We say that a $k$-homogeneous smooth valuation $\phi\in\Val_k^\infty$ is {\it primitive} if
\begin{align}
\phi\cdot\mu_1^{n-2k+1}=0,
\end{align}
and denote the subspace of all such valuations by $\pVal_k$. We shall also use the following natural notation: $\pVal_k^s=\pVal_k\cap\Val^{s,\infty}$, $s=0,1$. In these terms, an easy consequence of the hard Lefschetz theorem is a version of the {\it Lefschetz decomposition} for valuations. Namely,
\begin{corollary}
\label{cor:LD}
For any $0\leq k\leq \lfloor \frac n2\rfloor$, one has
\begin{align}
\label{eq:LD}
\Val_k^\infty=\bigoplus_{j=0}^k\mu_1^{k-j}\cdot\pVal_j,
\end{align}
and consequently,
\begin{align}
\label{eq:LDpar}
\Val_k^{s,\infty}=\bigoplus_{j=0}^k\mu_1^{k-j}\cdot\pVal_j^s,\quad s=0,1.
\end{align}
\end{corollary}

\begin{proof}
Since the multiplication by $\mu_1$ commutes with the action of $-\id$, \eqref{eq:LDpar} is an immediate consequence of \eqref{eq:LD} in fact. Let us thus show \eqref{eq:LD}. The case $k=0$ is trivial so assume otherwise. Obviously, $\mu_1\cdot\Val_{k-1}^\infty+\pVal_k\subset\Val_k^\infty$. The opposite inclusion follows from the surjectivity part of the hard Lefschetz theorem (Corollary \ref{cor:HLproduct}) in degree $k-1$: For any $\phi\in\Val_k^\infty$, there exists $\psi\in\Val_{k-1}^\infty$ with $\mu_1^{n-2k+2}\cdot\psi=\mu_1^{n-2k+1}\cdot\phi$. Then $\phi=\mu_1\cdot\psi+(\phi-\mu_1\cdot\psi)$ yields the desired decomposition. Finally, if $\phi\in\pVal_k\cap\,\mu_1\cdot\Val_{k-1}^\infty$, then $0=\mu_1^{n-2k+1}\cdot\phi=\mu_1^{n-2k+2}\cdot\psi$ for some $\psi\in\Val_{k-1}^\infty$ which must be zero by the injectivity part of the hard Lefschetz theorem. All in all, we have $\Val_k^\infty=\mu_1\cdot\Val_{k-1}^\infty\oplus\pVal_k$ from which the rest follows easily by induction.
\end{proof}

In order to proceed to discuss the Hodge-Riemann relations in the algebra $\Val^\infty$, let us recall from the introduction that the {\it Alesker-Hodge-Riemann pairing} $Q\maps{\Val_k^\infty\times\Val_k^\infty}{\Val_n^\infty\cong \CC}$ is given by
\begin{align}
\label{eq:AHRP1}
Q(\phi,\psi)=\b\phi\cdot\psi\cdot\mu_1^{n-2k}.
\end{align}
Observe that if $j$ is fixed, the signature of $Q$ is on each subspace $\mu_1^{i}\cdot\pVal_j$, $0\leq i\leq k-j$ the same. The Lefschetz decomposition therefore implies that the signature on the primitive subspaces $\pVal_j$, $0\leq j\leq k$ determines the signature on $\Val_k^\infty$. The same of course applies to $\pVal_j^s$ and $\Val_k^{s,\infty}$.

\section{Hodge-Riemann relations for even valuations}
\label{s:even}

In the section to follow, the Hodge-Riemann relations are proved for even smooth valuations. Our proof relies on representing the valuations in question by smooth functions, and the involved algebraic structures in terms of the Radon and cosine transforms, as summarized in \S\ref{ss:even}. This allows us to deduce the positivity  of \eqref{eq:HReven} from the sign of the transform eigenvalues. Recall that the notation is kept from, in particular, \S\ref{ss:fGr} and \S\ref{ss:even}.

Let $0\leq k\leq \lfloor\frac n2\rfloor$. Since the Lefschetz map, i.e., the multiplication by $\mu_1$,  is $\SO(n)$ equivariant, it follows at once from the Lefschetz decomposition \eqref{eq:LDpar} that the $\SO(n)$ module $P_k^0$ decomposes into irreducible subspaces corresponding to the following highest weights:
\begin{align}
\Pi_k^0=\left\{(m_1,\dots,m_k,0,\dots,0)\in\Lambda_k^0\mid \norm{m_2}\leq2,m_k\neq0\right\}.
\end{align}
Since the Crofton map is equivariant as well, this together with Theorem \ref{thm:AB} means that
\begin{align}
\Cr^{-1}(\pVal_k^0)=C^\infty(\Grass_k)\cap\widehat{\bigoplus_{\lambda\in\Pi_k^0}}\calH_\lambda.
\end{align}

Later on, we shall need to control the sign of the Radon-transform eigenvalues corresponding to the highest weights $\Pi_k^0$. To this end,

\begin{lemma}
\label{lem:signR}
Let $1\leq k\leq \lfloor \frac n2\rfloor$ and $\lambda=(2m,2,\dots,2,\pm2,0,\dots,0)\in\Pi_k^0$. There exists $c_{k,m}>0$ such that
\begin{align}
\perp^*\circ \,R_{k,n-k}|_{\calH_\lambda}=c_{k,m}\, (-1)^{m-1+k} \id|_{\calH_\lambda}.
\end{align}
\end{lemma}

\begin{remark}
Observe that the last non-zero entry of $\lambda$ is the $k$-th one.
\end{remark}

\begin{proof}
By Schur's Lemma, there is $\gamma\in\CC$ such that 
\begin{align}
\perp^*\circ \,R_{k,n-k}|_{\calH_\lambda}=\gamma\, \id|_{\calH_\lambda}.
\end{align}
Let $e_j$ be the $j$-th element of the standard orthonormal basis of $\RR^n$. Consider the highest-weight vector $h_\lambda\in\calH_\lambda$ and $E_0\in\Grass_k$ spanned by the (orthonormal) basis $\{e_1,e_3,e_5,\dots,e_{2k-1}\}$. Then
\begin{align*}
A[1]=(1,0,\dots,0),
\end{align*}
and
\begin{align*}
A[k]=\id_k.
\end{align*}
Therefore,
\begin{align*}
h_\lambda(E_0)&=\det\left(A[1]A[1]^t\right)^{m-1}\,\det\left(A[k]A[k]^t\right)=1,
\end{align*}
and, consequently, one has
\begin{align}
\label{eq:gamma}
\gamma=\left[\perp^*\circ \,R_{k,n-k}(h_\lambda)\right](E_0)=\int_{\Grass_{k}(E_0^\perp)}h_\lambda(F)\,d F.
\end{align}
Consider an arbitrary $F\in\Grass_k(E_0^\perp)$. Any orthonormal basis of $F$ must be of the following form:
\begin{align*}
\begin{pmatrix}0\\x_{2,1}\\0\\x_{4,1}\\\vdots \\0\\x_{2k,1}\\ *\end{pmatrix},\dots,\begin{pmatrix}0\\x_{2,k}\\0\\x_{4,k}\\\vdots \\0\\x_{2k,k}\\ *\end{pmatrix}.
\end{align*}
Here `$*$' stands for the remaining part of a vector which is insignificant to us. For such a basis,
\begin{align*}
A[1]=\sqrt{-1}(x_{2,1},x_{2,2},\dots,x_{2,k})\in\CC^{1\times k},
\end{align*}
and
\begin{align*}
A[k]=\sqrt{-1}\,X,\quad\text{where }X=\begin{pmatrix}x_{2,1}&\cdots&x_{2,k}\\\vdots&&\vdots\\x_{2k,1}&\cdots&x_{2k,k}\end{pmatrix}\in\RR^{k\times k}.
\end{align*}
Hence, for some $c_{k,m,F}\geq 0$,
\begin{align*}
h_\lambda(F)&=\det\left(A[1]A[1]^t\right)^{m-1}\,\det\left(A[k]A[k]^t\right) \\
&=(-1)^{m-1}\left(\sum_{j=1}^k(x_{2,j})^2\right)^{m-1}(-1)^k\,(\det X)^2\\
&=c_{k,m,F}\,(-1)^{m-1+k}.
\end{align*}
Plugging this into \eqref{eq:gamma} and taking into account that $h_\lambda$ is continuous and that for
\begin{align*}
F_0=\RR\{e_2,e_4,\dots,e_{2k}\}\in\Grass_k(E_0^\perp)
\end{align*}
one has
\begin{align*}
h_\lambda(F_0)=(-1)^{m-1+k}\neq0,
\end{align*}
we finally obtain that, for some $c_{k,m}>0$,
\begin{align*}
\gamma=c_{k,m}\,(-1)^{m-1+k}.
\end{align*}
\end{proof}

As no similarly simple argument is known to us to prove a counterpart statement for the cosine transform, we deduce it as a consequence of Lemma \ref{lem:multipT}.
\begin{lemma}
\label{cor:signT}
Let $1\leq k\leq \lfloor \frac n2\rfloor$ and $\lambda=(2m,2,\dots,2,\pm2,0,\dots,0)\in\Pi_k^0$. There exists $c_{k,m}>0$ such that
\begin{align}
\label{eq:signT}
\perp^*\circ \,T_{n-k}\,\circ\perp^*|_{\calH_\lambda}=c_{k,m}\,(-1)^{m-1}\, \id|_{\calH_\lambda}.
\end{align}
\end{lemma}
\begin{proof}
We observe that $\perp^*\circ \,T_{n-k}\,\circ\perp^*=T_k$ and apply Lemma \ref{lem:multipT} for $m_1=m$ and $m_2=\dots=m_{k-1}=\norm{m_k}=1$. In this case,
\begin{align*}
\prod_{j=1}^k\frac{\left(1+\frac j2-\norm{m_j}\right)_{\norm{m_j}}}{\left(1+\frac n2-\frac j2\right)_{\norm{m_j}}}=\frac{\left(\frac 32-m\right)_{m}}{\left(\frac12+\frac n2\right)_{m}}\,\prod_{j=2}^{k}\frac{\frac j2}{1+\frac n2-\frac j2}.
\end{align*}
Since
\begin{align*}
\frac{1}{\left(\frac12+\frac n2\right)_{m}}\,\prod_{j=2}^{k}\frac{\frac j2}{1+\frac n2-\frac j2}>0
\end{align*}
and, for some $c_m>0$,
\begin{align*}
\left(\frac 32-m\right)_{m}=\left(\frac 32-m\right)\,\left(\frac 32-m+1\right)\cdots\left(-\frac12\right)\,\frac12=c_m\, (-1)^{m-1},
\end{align*}
\eqref{eq:signT} follows.
\end{proof}

We can now proceed to the proof of our first main result.
\begin{proof}[Proof of Theorem \ref{thm:HReven}]
First, let $\lambda=(2m,2,\dots,2,\pm2,0,\dots,0)\in\Pi_k^0$. Since $\perp^2=\id$, Lemma \ref{lem:signR} and Lemma \ref{cor:signT} together imply
\begin{align}
\label{eq:signTR}
T_{n-k}\circ R_{k,n-k}|_{\calH_\lambda}=c_{k,m}\, (-1)^k\, \perp^*|_{\calH_\lambda},
\end{align}
for some $c_{k,m}>0$.

Now, take any non-zero $\phi\in\pVal_k^0$ and consider the $L^2$-orthogonal decomposition
\begin{align}
\label{eq:fphi}
f_\phi=\sum_{\lambda\in\Pi_k^0}f_\phi^{(\lambda)},\quad f_\phi^{(\lambda)}\in\calH_\lambda.
\end{align}
We shall use Lemma \ref{lem:pd+} for $\varphi=\FF(\b\phi)$ and $\psi=\phi\cdot \mu_{n-2k}$. Recall also that $\mu_1^{n-2k}=c_k\mu_{n-2k}$, for some $c_k>0$, according to Theorem \ref{thm:AP} (d). All in all, we have
\begin{align*}
Q(\phi,\phi)&=\b\phi\cdot\phi\cdot\mu_1^{n-2k}\\
&=c_k\,\b\phi\cdot(\phi\cdot\mu_{n-2k})\\
\overset{\eqref{eq:pd+}}&{=}c_k\int_{\Grass_{n-k}}f_{\FF(\b\phi)}(E)\,\Klain_{\phi\cdot \mu_{n-2k}}(E)\,d E\\
\overset{\eqref{eq:Klaintimes}}&{=}\tilde c_k\int_{\Grass_{n-k}}f_{\FF(\b\phi)}(E)\,\left[T_{n-k}\circ R_{k,n-k}(f_\phi)\right](E)\,d E\\
\overset{\eqref{eq:FFperp}}&{=}\tilde c_k\int_{\Grass_{n-k}}\b{f_\phi(E^\perp)}\,\left[T_{n-k}\circ R_{k,n-k}(f_\phi)\right](E)\,d E\\
\overset{\eqref{eq:fphi}}&{=}\tilde c_k\sum_{\lambda',\lambda\in\Pi_k^0}\int_{\Grass_{n-k}}\b{f^{(\lambda')}_\phi(E^\perp)}\,\left[T_{n-k}\circ R_{k,n-k}\left(f^{(\lambda)}_\phi\right)\right](E)\,d E\\
\overset{\eqref{eq:signTR}}&{=}(-1)^k\sum_{\lambda',\lambda\in\Pi_k^0}c_{k,\lambda}\int_{\Grass_{n-k}}\b{f^{(\lambda')}_\phi(E^\perp)}\, f^{(\lambda)}_\phi(E^\perp)\,d E\\
&=(-1)^k\sum_{\lambda',\lambda\in\Pi_k^0}c_{k,\lambda}\int_{\Grass_{k}}\b{f^{(\lambda')}_\phi(F)}\, f^{(\lambda)}_\phi(F)\,d F\\
&=(-1)^k\sum_{\lambda\in\Pi_k^0}c_{k,\lambda}\int_{\Grass_{n-k}}\sqnorm{f^{(\lambda)}_\phi(F^\perp)}\,d F,
\end{align*}
for some $c_k,\tilde c_k,c_{k,\lambda}>0$. Since $f_\phi\neq0$, there is $\lambda\in\Pi_k^0$ with $f_\phi^{(\lambda)}\neq0$, and consequently,
\begin{align*}
(-1)^k\,Q(\phi,\phi)>0,
\end{align*}
as desired.
\end{proof}

\section{Hodge-Riemann relations for 1-homogeneous valuations}

The purpose of this section is to prove the Hodge-Riemann relation for smooth valuations of degree 1, regardless of parity. Using a different, yet analogous argumentation as in the even case considered in \S\ref{s:even}, the proof is based on a functional representation of these valuations, in particular, on Alesker's characterization theorem (Lemma \ref{lem:spherical1}) and recent results on spherical valuations due to Bernig and Hug (Lemma \ref{lem:spherical}). Recall that the notation is kept from \S\ref{ss:spherical}.

\begin{proposition}
Let $1\leq k\leq n$. There is $c_k>0$ such that for any $f\in C^\infty(S^{n-1})$, one has
\begin{align}
\label{eq:convLmukf}
\mu_{n-1}*\mu_{k,f}=c_k\,\mu_{k-1,f}.
\end{align}
\end{proposition}

\begin{proof}
Using the formulas \eqref{eq:convLefschetz} and \eqref{eq:SteinerSk}, respectively, we indeed have
\begin{align*}
\mu_{n-1}*\mu_{k,f}(K)&=\frac12\int_{S^{n-1}}f(y)\,\frac{d}{d\lambda}\bigg|_{\lambda=0}d S_k(K+\lambda D,y)\\
&=c_k\int_{S^{n-1}}f(y)\,d S_{k-1}(K,y)\\
&=c_k\,\mu_{k-1,f}(K)
\end{align*}
for any $K\in\calK$ and some constant $c_k>0$.
\end{proof}

\begin{corollary}
Let $0\leq k\leq n-1$ and $q\in\NN_0$ with $q\neq1$. There is $c_{k,q}>0$ such that for any $f\in \calS_q$, one has
\begin{align}
\label{eq:productLmukf}
\mu_1\cdot\mu_{k,f}=c_{k,q}\,\mu_{k+1,f}.
\end{align}
\end{corollary}

\begin{proof}
According to the previous proposition and Theorem \ref{thm:AFT}, we have
\begin{align*}
\mu_1\cdot\mu_{k,f}&= \FF^{-1}\left(\FF\mu_1*\FF\mu_{k,f}\right)\\
\overset{\eqref{eq:FFmukf}}&{=}\tilde c_{k,q}\,(\sqrt{-1})^q\,\FF^{-1}\left(\mu_{n-1}*\mu_{n-k,f}\right)\\
\overset{\eqref{eq:convLmukf}}&{=}\hat c_{k,q}\,(\sqrt{-1})^q\,\FF^{-1}\mu_{n-k-1,f}\\
\overset{\eqref{eq:FFmukf}}&{=}c_{k,q}\,\mu_{k+1,f},
\end{align*}
for some constants $\tilde c_{k,q},\hat c_{k,q},c_{k,q}>0$.
\end{proof}

We are now in position to prove the second main result of this article.
\begin{proof}[Proof of Theorem \ref{thm:HR1}]
Take an arbitrary non-zero $\phi\in\pVal_1^s$. First, according to Lemma \ref{lem:spherical1}, $\phi=\mu_{1,f}$ for some $f=\sum_q f^{(q)}$, where $f^{(q)}\in\calS_q$, and the sum extends over all $q\in\NN_0$ with $q\equiv s\mod 2$ and $q\neq 1$. That $\phi$ is primitive means that
\begin{align*}
0&=\mu_{1,f}\cdot\mu_{n-1}\\
\overset{\eqref{eq:muk1}}&{=}\sum_q\mu_{1,f^{(q)}}\cdot\mu_{n-1,1}\\
\overset{\eqref{eq:pdmufg}}&{=}\mu_{1,f^{(0)}}\cdot\mu_{n-1,1}\\
\overset{\eqref{eq:pdmufg}}&{=} c\int_{S^{n-1}}f^{(0)}(y)\,dy,
\end{align*}
for some $c>0$. Since $f^{(0)}\in\calS_0$ is constant, this implies $f^{(0)}=0$. So we in fact have $f=\sum_q f^{(q)}$ with $q\equiv s\mod 2$ and $q\geq 2$. Then,
\begin{align*}
Q(\phi,\phi)&=\b{\mu_{1,f}}\cdot\mu_{1,f}\cdot\mu_1^{n-2}\\
&=\sum_{q,r}\mu_{1,\b{f^{(q)}}}\cdot\mu_{1,f^{(r)}}\cdot\mu_1^{n-2}\\
\overset{\eqref{eq:productLmukf}}&{=}\sum_{q,r}c_r\,\mu_{1,\b{f^{(q)}}}\cdot\mu_{n-1,f^{(r)}}\\
\overset{\eqref{eq:pdmufg}}&{=}\sum_{q}\tilde c_q\,(-1)^q\left(1-\frac{q(n+q-2)}{n-1}\right)\int_{S^{n-1}}\norm{f^{(q)}(y)}^2dy,
\end{align*}
for some $c_r,\tilde c_q>0$. Observe that $(-1)^q=(-1)^s$ and
\begin{align*}
1-\frac{q(n+q-2)}{n-1}=\frac{(1-q)(n-1+q)}{n-1}
\end{align*}
is always negative for $q,n\geq2$. Finally, since $\phi\neq0$, there is $q_0$ with $f^{(q_0)}\neq0$, and therefore
\begin{align*}
(-1)^s\, Q(\phi,\phi)<0.
\end{align*}
\end{proof}

\begin{remark}
Let us point out that the method of the current section is, alas, not sufficient to obtain even partial results in higher degrees as all spherical valuations then obviously lie in the image of the Lefschetz map, i.e., they are not primitive unless trivial.
\end{remark}

\section{Hodge-Riemann relations in the language of the Bernig-Fu convolution}

We now reformulate both the proved and the conjectured Hodge-Riemann relations in terms of the other canonical multiplicative structure on $\Val^\infty$. As we shall see, the setting of the Bernig-Fu convolution is more suitable for applications to isoperimetric inequalities, and leads the way to further generalizations.

Let $0\leq k\leq \lfloor \frac n2\rfloor$. We say that a valuation $\phi\in\Val_{n-k}^\infty$ of co-degree $k$ is {\it co-primitive} if
\begin{align}
\label{eq:co-primitive}
\phi*\mu_{n-1}^{*(n-2k+1)}=0,
\end{align}
and denote the subspace of all such by $\cpVal_{n-k}$. We also write $\cpVal_{n-k}^s=\cpVal_{n-k}\cap\Val^{s,\infty}$, $s=0,1$. Observe that Theorem \ref{thm:AFT} yields at once
\begin{align}
\label{eq:cpVal}
\cpVal^s_{n-k}=\FF\pVal_k^s.
\end{align}
In analogy with \eqref{eq:AHRP1} and in agreement with the introduction, we define the {\it Bernig-Fu-Hodge-Riemann pairing} $\wt Q\maps{\Val_{n-k}^\infty\times\Val_{n-k}^\infty}{\CC}$ by
\begin{align}
\label{eq:Qtilda}
\wt Q(\phi,\psi)=\b\phi*\psi*\mu_{n-1}^{*(n-2k)}.
\end{align}

The key to understand the (perhaps unexpected) dependence of \eqref{eq:HR} on the parity of a valuation turns out to be the following observation, whose proof is due to S.~Alesker:

\begin{lemma}
\label{lem:Freal}
Let $0\leq k\leq n$ and let $\phi\in\Real\Val_k^\infty$ be a real-valued valuation.
\begin{enuma}
\item If $\phi$ is even, then $\FF\phi$ is real valued.
\item If $\phi$ is odd, then $\FF\phi$ is purely imaginary valued.
\end{enuma}
\end{lemma}

\begin{proof}
\quad
\begin{enuma}
\item See \cite{Alesker11}, Theorem 5.4.1 (3).
\item Because $\FF^2=-\id$ on $\Val^{1,\infty}$, we may assume $k\leq\lfloor\frac n2\rfloor$. First, the case $k=0$ is trivial since $\Val_0^{1,\infty}=\{0\}$. Second, for $k=1$ the claim follows at once from \eqref{eq:FFmukf} and Lemma \ref{lem:spherical1}. Finally, assume $k\geq 2$. Take any $\phi\in\Real\Val_1^{1,\infty}$ and $\psi\in\Real\Val_{k-1}^{0,\infty}$. By what has been already shown, $\FF\phi$ is purely imaginary while $\FF\psi$ is real. Consequently, $\FF(\phi\cdot\psi)=\FF\phi*\FF\psi$ is purely imaginary by \eqref{eq:convDef}. The irreducibility theorem (Theorem \ref{thm:IT}) holds verbatim for $\Real\Val_l^{s,\infty}$ as the complexification of a dense invariant real subspace is a dense invariant subspace of $\Val_l^{s,\infty}$. Hence
\begin{align*}
\linspan_\RR\left\{\phi\cdot\psi\mid \phi\in\Real\Val_1^{1,\infty},\psi\in\Real\Val_{k-1}^{0,\infty}\right\}\subset\Real\Val_k^{1,\infty},
\end{align*}
which is $\GL(n,\RR)$ invariant by \eqref{eq:APequivariance}, and obviously non-trivial (see Lemma \ref{lem:spherical1} and Corollary \ref{cor:HLproduct}), is dense. Now the claim of the lemma easily follows from $\CC$-linearity and continuity of $\FF$.
\end{enuma}
\end{proof}

\begin{proposition}
\label{pro:Qtilda}
For any $\phi\in\Val_k^{\infty}$, one has
\begin{align}
\wt Q({\FF\phi},\FF\phi)=Q({\phi_0},\phi_0)-Q({\phi_1},\phi_1),
\end{align}
where $\phi_0$ and $\phi_1$ are the even and odd part of $\phi$, i.e., $\phi=\phi_0+\phi_1$ and $\phi_s\in\Val_k^{\infty,s}$, $s=0,1$.
\end{proposition}

\begin{proof}
Observe that $\Val_k^{0,\infty}$ and $\Val_k^{1,\infty}$ are orthogonal with respect to $Q$, and similarly $\Val_{n-k}^{0,\infty}$ and $\Val_{n-k}^{1,\infty}$ with respect to $\wt Q$, as $\mu_l\in\Val_l^{0,\infty}$ is even for all $0\leq l\leq n$. Further, it is an immediate consequence of Lemma \ref{lem:Freal} that
\begin{align*}
\Real\FF\phi_0&=\FF\Real\phi_0,\\
\Imag\FF\phi_0&=\FF\Imag\phi_0,
\end{align*}
and
\begin{align*}
\Real\FF\phi_1&=\sqrt{-1}\,\FF\Imag\phi_1,\\
\Imag\FF\phi_1&=-\sqrt{-1}\,\FF\Real\phi_1.
\end{align*}
From these facts and Theorem \ref{thm:AFT} one infers
\begin{align*}
\wt Q({\FF\phi},\FF\phi)&=\wt Q({\FF\phi_0},\FF\phi_0)+\wt Q({\FF\phi_1},\FF\phi_1)\\
&=\left((\Real\FF\phi_0)^{*2}+(\Imag\FF\phi_0)^{*2}+(\Real\FF\phi_1)^{*2}+(\Imag\FF\phi_1)^{*2}\right)*\mu_{n-1}^{*(n-2k)}\\
&=\left((\FF\Real\phi_0)^{*2}+(\FF\Imag\phi_0)^{*2}-(\FF\Imag\phi_1)^{*2}-(\FF\Real\phi_1)^{*2}\right)*\mu_{n-1}^{*(n-2k)}\\
&=\FF^{-1}\bigg[\left((\FF\Real\phi_0)^{*2}+(\FF\Imag\phi_0)^{*2}-(\FF\Imag\phi_1)^{*2}-(\FF\Real\phi_1)^{*2}\right)*\mu_{n-1}^{*(n-2k)}\bigg]\\
&=\left((\Real\phi_0)^{2}+(\Imag\phi_0)^{2}-(\Imag\phi_1)^{2}-(\Real\phi_1)^{2}\right)\cdot\mu_{1}^{n-2k}\\
&=Q({\phi_0},\phi_0)-Q({\phi_1},\phi_1),
\end{align*}
where $\Val_0^\infty$ and $\Val_n^\infty$ are again identified with $\CC$ via $\chi$ and $\vol_n$, respectively.
\end{proof}

Proposition \ref{pro:Qtilda} together with \eqref{eq:cpVal} implies at once that Conjecture \ref{con:HR} is indeed equivalent to Conjecture \ref{con:HRc}. Along the same lines, an equivalent formulation of Theorems \ref{thm:HReven} and \ref{thm:HR1}, respectively, is the following:
\begin{theorem}
\label{thm:HRc}
Conjecture \ref{con:HRc} is true under each of the following additional assumptions:
\begin{enuma}
\item $\phi$ is even,
\item $k=1$.
\end{enuma}
\end{theorem}

\section{An isoperimetric-type inequality}

In this section, an application of one of the proven cases of the Hodge-Riemann relations is discussed: It turns out that the purely algebraic statement of Theorem \ref{thm:HRc} (b) can be used to deduce an inequality of geometric type, namely, the first one of the isoperimetric inequalities \eqref{eq:II}. A special case (assuming $n=2$ and $K=-K$) of the argument we use was communicated to us by S. Alesker.

\begin{corollary}
\label{cor:iso12}
For any $K\in\calK_+^\infty$, one has
\begin{align}
\label{eq:iso12}
V(K,D[n-1])^2\geq V(K,K,D[n-2])\cdot\vol_n(D).
\end{align}
\end{corollary}

\begin{remark}
It follows at once from \eqref{eq:IV} that \eqref{eq:iso12} is indeed equivalent to
\begin{align}
\left(\frac{\mu_1(K)}{\mu_1(D)}\right)^2\geq\frac{\mu_2(K)}{\mu_2(D)}.
\end{align}
In another terminology, \eqref{eq:iso12} is a special case of {\it Minkowski’s second inequality} (see \cite{Schneider2014}, p. 382), and obviously a special case of the Aleksandrov-Fenchel inequality \eqref{eq:AF}.

\end{remark}

\begin{proof}
Consider the following valuation:
\begin{align*}
\eta=V(\Cdot[n-1],K)-\frac{V(K,D[n-1])}{\vol_n(D)}\, V(\Cdot[n-1],D).
\end{align*}
Clearly, $\eta\in\Val_{n-1}^\infty$. Further, $\eta\in\cpVal_{n-1}$ in fact since it follows from \eqref{eq:IV} and \eqref{eq:MVconv} that
\begin{align*}
\eta*\mu_{n-1}^{*(n-1)}&=c\,\eta*V(\Cdot,D[n-1])=\tilde c\left[V(K,D[n-1])-\frac{V(K,D[n-1])}{\vol_n(D)}\, V(D[n])\right]=0,
\end{align*}
for some $c,\tilde c>0$.
Consequently, according to Theorem \ref{thm:HRc} (b) and formulas \eqref{eq:IV} and \eqref{eq:MVconv}, there are $b,\tilde b, \hat b>0$ such that 
\begin{align*}
0&\geq \wt Q(\eta,\eta)\\
&=\eta*\eta*\mu_{n-1}^{*(n-2)}\\
&=b\,\eta*\eta*V(\Cdot[2],D[n-2]) \\
&=\tilde b\left[V(\Cdot[n-2],K[2])-2\frac{V(K,D[n-1])}{\vol_n(D)}\,V(\Cdot[n-2],K,D)\right.\\
&\qquad\qquad\left.+\frac{V(K,D[n-1])^2}{\vol_n(D)^2}\,V(\Cdot[n-2],D[2])\right]*V(\Cdot[2],D[n-2])\\
&=\hat b\left[V(K[2],D[n-2])-2\frac{V(K,D[n-1])}{\vol_n(D)}\,V(K,D[n-1])+\frac{V(K,D[n-1])^2}{\vol_n(D)^2}\,V(D[n])\right]\\
&=\hat b\left[V(K[2],D[n-2])-\frac{V(K,D[n-1])^2}{\vol_n(D)}\right],
\end{align*}
which is clearly equivalent to \eqref{eq:iso12}.
\end{proof}

\section{Mixed Hodge-Riemann relations and the Aleksandrov–Fenchel inequality}

We conclude by generalizing the arguments of the previous section to the conjectured mixed version of the Hodge-Riemann relations. In particular, this allows us to show that Conjecture \ref{con:mixed} yields the Aleksandrov-Fenchel inequality as a special case.

Assume, at first, $k=0$. It follows at once from \eqref{eq:MVconv} that in this case Conjecture \ref{con:mixed} is equivalent to the statement that for any $K_1,\dots,K_n\in\K_+^\infty$, one has
\begin{align}
V(K_1,\dots,K_n)>0,
\end{align}
which is a well-known, yet non-trivial fact (see \cite{Schneider2014}, Theorems 5.1.7 and 5.1.8).

More interesting, however, is the next case:

\begin{corollary}
Conjecture \ref{con:mixed} for $k=1$ implies the Aleksandrov-Fenchel inequality \eqref{eq:AF} for convex bodies from the class $\calK_+^\infty$.
\end{corollary}

\begin{remark}
The full generality of the Aleksandrov-Fenchel inequality, namely, the extension of its validity to the class $\K$, is then achieved by the standard limiting argument (cf. Aleksandrov's second proof \cite{Aleksandrov38}).
\end{remark}

\begin{proof}
The same argumentation is valid here as that of the proof of Corollary \ref{cor:iso12}. Namely, take any $K_1,\dots,K_n\in\calK_+^\infty$ and consider the following `mixed' version of the valuation $\eta$ examined therein:
\begin{align*}
\xi=V(\Cdot[n-1],K_1)-\frac{V(K_1,K_2,K_3,\dots,K_n)}{V(K_2,K_2,K_3,\dots,K_n)}\,V(\Cdot[n-1],K_2)\in\Val_{n-1}^\infty.
\end{align*}
Recall that $V(K_2,K_2,K_3,\dots,K_n)>0$ and $\xi$ is thus well defined. Using \eqref{eq:MVconv}, we easily compute
\begin{align*}
\xi*V(\Cdot[n-1],K_2)*\cdots*V(\Cdot[n-1],K_n)=0.
\end{align*}
Hence, part (b) of Conjecture \ref{con:mixed} for $k=1$ implies
\begin{align}
\label{eq:geqzeta}
0\geq\xi*\xi*V(\Cdot[n-1],K_3)*\cdots*V(\Cdot[n-1],K_n).
\end{align}
By precisely the same considerations as above, i.e., using the formula \eqref{eq:MVconv} repeatedly, we find that the right-hand side of \eqref{eq:geqzeta} is a positive multiple of
\begin{align*}
V(K_1,K_1,K_3,\dots,K_n)-\frac{V(K_1,K_2,K_3,\dots,K_n)^2}{V(K_2,K_2,K_3,\dots,K_n)},
\end{align*}
and \eqref{eq:geqzeta} is therefore equivalent to \eqref{eq:AF}.
\end{proof}

\bibliography{bib/books,bib/papers,bib/theses}
\bibliographystyle{abbrv}

\end{document}